\documentclass[12pt]{article}
\usepackage{booktabs}
\usepackage{caption}
\usepackage{mathrsfs}
\usepackage{amsmath}
\usepackage{amsfonts,amsthm,amssymb,mathrsfs,bbding}
\usepackage{float}
\usepackage{graphics,epstopdf}
\usepackage{graphics,multicol}
\usepackage{graphicx}
\usepackage{color}
\usepackage{enumerate}
\usepackage{caption}
\usepackage{rotating}
\usepackage{lscape}
\usepackage{longtable}
\allowdisplaybreaks[4]
\usepackage{tabularx}
\usepackage[colorlinks=true,anchorcolor=blue,filecolor=blue,linkcolor=blue,urlcolor=blue,citecolor=blue]{hyperref}
\usepackage{extarrows}
\usepackage{cite}
\usepackage{latexsym,bm}
\usepackage{mathtools}
\evensidemargin=\oddsidemargin\topmargin=-1.5cm

\newtheorem{thm}{Theorem}

\newtheorem{conj}{Conjecture}[section]
\newtheorem{claim}{Claim}

\addtocounter{section}{0}
\begin{document}
\title{Maximum degree and spectral radius of graphs in terms of size\footnote{
This work is supported by National Natural Science Foundation of China (Nos. 12171154, 1211101361)
and
the China Postdoctoral Science Foundation (No. 2021M691671).}}

\author{Zhiwen Wang$^{a}$, \ \ Ji-Ming Guo$^{b}$\footnote{Corresponding author.
\newline{\it \hspace*{5mm}Email addresses:}
jimingguo@hotmail.com (J. Guo), walkerwzw@163.com (Z. Wang).}\\[2mm]
\small $^a$School of Mathematical Sciences, NanKai University, Tianjin, 300071, China\\
\small $^b$Department of Mathematics, East China University of Science and Technology, \\
\small Shanghai, 200237, China}
\date{}
\maketitle
{\flushleft\large\bf Abstract:}
Research on the relationship of the (signless Laplacian) spectral radius of a graph with its structure properties
is an important research project in spectral graph theory.
Denote by $\rho(G)$ and $q(G)$ the spectral radius and the signless Laplacian spectral radius of a graph $G$, respectively.
Let $k\ge 0$ be a fixed integer and $G$ be a graph of size $m$ which is large enough.
We show that if $\rho(G)\ge\sqrt{m-k}$,
then $C_4\subseteq G$ or $K_{1,m-k}\subseteq G$.
Furthermore, we prove that if $q(G)\ge m-k$, then $K_{1,m-k}\subseteq G$.
Both these two results extend some known results.
\begin{flushleft}
\textbf{Keywords:} Spectral radius; Maximum degree; Size
\end{flushleft}
\textbf{AMS Classification:} 05C50; 05C35

\section{Introduction}

\indent \, \quad Graphs considered in the paper are simple and undirected.
For a graph $G$, let $\rho(G)$ be the spectral radius of its adjacency matrix $A(G)$,
and $q(G)$ be the spectral radius of its signless Laplacian matrix $Q(G)$.
From Perron--Frobenius theorem,
for a connected graph $G$, the adjacency (resp., signless Laplacian) spectral radius of $G$
is the maximum modulus of its adjacency (resp., signless Laplacian) eigenvalues.
In general, we call $\rho(G)$ the spectral radius of $G$,
and $q(G)$ the signless Laplacian spectral radius of $G$.

It is well known that the structure properties and parameters of graphs
have close relationship with eigenvalues of graphs.
During this recent thirty years, the (signless Laplacian) spectral radius among graphs with described structures properties
has attracted considerable attention.

A graph $G$ is defined to be $H$-free if $G$ does not contain $H$ as a subgraph.
As a spectral version of extremal graph theory,
Nikiforov \cite{Niki10} posed a spectral Tur\'an type problem that
 \textit{what is the maximal spectral radius $\rho(G)$ among $H$-free graphs $G$ of order $n$?}
This problem is also known as the Brualdi--Solheid--Tur\'an type problem
and has been investigated in much literature for some special graphs $H$, for which one can refer to
clique \cite{Niki07}, book \cite{ZL}, friendship \cite{CFTZ20} and references therein.

Recently, replacing the order $n$ by the size $m$,
a perspective to spectral Tur\'an type problem in terms of the size has received much research.
This problem asks that
 \textit{what is the maximal spectral radius $\rho(G)$ among $H$-free graphs $G$ of size $m$?}
To the knowledge of us, the history of studying this problem may be dated back at least to Nosal's
theorem \cite{N70} in 1970.
Up to now, there is few graph $H$ such that the maximal spectral radius $\rho(G)$
among $H$-free graphs of size $m$ has been determined.
Some relevant conclusions have been obtained in the past two decades.
Nikiforov in \cite{Niki02} extended Nosal's theorem from triangles to clique, and answered a conjecture
by Zhai, Lin and Shu \cite{ZLS} about books in \cite{Niki-arxiv}.
For more detailed results, we refer one to \cite{LNW, ZLS}.

Here we pay our main attention to the spectral Tur\'an type problem on quadrilaterals and stars in terms of the size.
Let $K_{1,n-1}$ and $C_n$ be a star and a cycle on $n$ vertices, respectively.
Denote by $K_{1,n-1}+e$ the graph by inserting an edge to the independent set of $K_{1,n-1}$,
and denote by $K_{1,n-1}^e$ the graph by attaching a pendent vertex to a pendent vertex of $K_{1,n-1}$.

In \cite{Niki09},
Nikiforov determined the maximum spectral radius among all $C_4$-free graphs
of size $m$.

\begin{thm}\textsc{\cite{Niki09}}\label{thm:Niki}
  Let $G$ be a graph of size $m\ge 9$. If $\rho(G)>\sqrt{m}$, then
  $C_4\subseteq G$.
\end{thm}

Zhai and Shu \cite{ZS22} improved the result
in Theorem \ref{thm:Niki} for a non-bipartite connected graph by showing the following
theorem.

\begin{thm}\textsc{\cite{ZS22}}\label{thm:Zhai}
Let $G$ be a non-bipartite and connected graph of size $m\ge 26$.
If $\rho(G)\ge \rho(K_{1,m-1}+e)$, then
$C_4\subseteq G$ unless $G$ is $K_{1,m-1}+e$.
\end{thm}

Recently, Wang \cite{W22+} provided a generalization of Theorems \ref{thm:Niki} and \ref{thm:Zhai}.
\begin{thm}\textsc{\cite{W22+}}\label{thm:Wang}
  Let $G$ be a graph of size $m\ge 27$. If $\rho(G)\ge \sqrt{m-1}$, then
  $C_4\subseteq G$ unless $G$ is one of these graphs (with possibly isolated vertices):
  $K_{1,m}$, $K_{1,m-1}+e$, $K_{1,m-1}^e$, or $K_{1,m-1}\cup P_2$.
\end{thm}

It is easy to check that $\rho(H)<\rho(K_{1,m})$ if $H\in \{K_{1,m-1}+e, K_{1,m-1}^e, K_{1,m-1}\cup P_2\}$ and $m\ge 27$.
This, together with Theorems \ref{thm:Niki} and \ref{thm:Wang}, indicates that
if $\rho(G)\ge\sqrt{m}$ for a graph $G$ of size $m\ge 27$, then $C_4\subseteq G$ unless $G$ is $K_{1,m}$.
Indeed, from Theorems \ref{thm:Niki} and \ref{thm:Wang}, if $m\ge 27$ and $\rho(G)\ge\sqrt{m-k}$ for $k=0$ or $1$,
then $C_4\subseteq G$ unless $K_{1,m-k}\subseteq G$.
Motivated by this, we hope to give a general result in terms of the value of $k$.

\begin{thm}\label{main-thm-1}
Let $k\ge 0$ be an integer and $G$ be a graph of size $m\ge\max\{(k^2+2k+2)^2+k+1,(2k+3)^2+k+1\}$.
If $\rho(G)\ge\sqrt{m-k}$, then $K_{1,m-k}\subseteq G$ or $C_4\subseteq G$.
\end{thm}

Next we turn our attention to study the relation of the maximum degree and the signless Laplacian spectral radius of a graph.
In a sense, the signless Laplacian matrix can significantly reveal the structure properties of graphs $G$
since $Q(G)$ consists of the adjacency matrix and the diagonal matrix of degree sequence.
For more details, readers are referred to \cite{OLAH10,LGW21},
and a series of survey by Cvetkovi\'c and Simi\'c \cite{CS09,CS10,CS10+}.

A signless spectral Tur\'an type version of extremal graph theory has been extensively studied by researchers.
Much literature studied the maximal signless Laplacian spectral $q(G)$ among $H$-free graphs in terms of order,
including triangles \cite{ZHG21}, cycles \cite{Y14, NY15} and linear forests \cite{CLZ20}.
There is few investigation on signless spectral Tur\'an type problem in terms of the size.

The topic we focus on is inspired from a theorem by Zhai, Xue and Lou \cite{ZXL20}, which can be viewed as
signless spectral Tur\'an type problem for stars in terms of the size.

\begin{thm}\label{thm:Zhai+Xue+Lou}\textsc{\cite{ZXL20}}
Let $G$ be a graph of size $m\ge 4$. If $G$ is a graph without isolated vertices, then
$q(G)\le m + 1$ with equality if and only if $G= K_{1,m}$.
\end{thm}

Theorem \ref{thm:Zhai+Xue+Lou} infers that if $q(G)\ge m+1$ for a graph of size $m$, then
$K_{1,m}\subseteq G$ (in fact, $G=K_{1,m}$ when $G$ has no isolate vertex).
We show the following result, which extends Theorem \ref{thm:Zhai+Xue+Lou}.

\begin{thm}\label{main-thm-2}
Let $k\ge 0$ be an integer and $G$ be a graph of size $m\ge\max\{\frac{1}{2}k^2+6k+3, 7k+25\}$.
If $q(G)\ge m-k+1$, then $K_{1,m-k}\subseteq G$.
\end{thm}

The rest of this paper is organized as follows.
Notations are introduced in Section \ref{section2}.
Proofs of Theorems \ref{main-thm-1} and \ref{main-thm-2} are presented in
Sections \ref{section2} and \ref{section3}, respectively.
In Section \ref{concluding work}, we propose a conjecture on the relation of maximum degree and
spectral radius of adjacency matrix in terms of order.

\section{Proof of Theorem \ref{main-thm-1}}\label{section2}

\indent \, \quad Before stating details of the proof, we shall introduce some terminologies and notations.
For a graph $G$, let $u$ be a vertex of $G$, and $S,T$ be two subsets of $V(G)$.
Then let $N_S(u)$ denote the set of neighbor of $u$ in $S$, and $d_S(u)$ be the cardinality of $N_S(u)$,
i.e., $d_S(u)=|N_S(u)|$.
Specially, if $S=V(G)$ then we omit the subscript $S$.
The minimum degree of $G$ is defined to be $\delta(G)=\min\{d(u):u\in V(G)\}$.
Let $G[S]$ be the subgraph of $G$ induced by $S$, then
denote $E(S)$ by the set of edges in $G[S]$ and $e(S)$ by the cardinality of $E(S)$.
Suppose that $S\cap T=\varnothing$.
Then we denote $e(S,T)$ by the number of edges with one endpoint in $S$ and another endpoint in $T$.

\vskip0.1in
\noindent
{\bf Proof of Theorem \ref{main-thm-1}.}
We prove Theorem \ref{main-thm-1} by way of contradiction.
Assume that there are graphs $H$ of size $m\ge\max\{(k^2+2k+2)^2+k+1,(2k+3)^2+k+1\}$ with $\rho(H)\ge\sqrt{m-k}$,
such that $K_{1,m-k}\not\subseteq H$ and $C_4\not\subseteq H$.
Let $G$ be a graph with the maximum spectral radius among graphs satisfying the above conditions.
Since adding/deleting isolated vertices to/from $G$ not changes the value of $\rho(G)$,
we can let $G$ contain no isolated vertices. For simplification, we write $\rho$ by $\rho(G)$.

Let $\emph{\textbf{x}}$ be a nonnegative eigenvector of $A(G)$ corresponding to $\rho$
with coordinate $x_i$ corresponding to the vertex $v_i$ of $G$.
Let $u^*$ be a vertex of $G$ with $x_{u^*}=\max\{x_i:v_i\in V(G)\}$, then
we have a partition $\{u^*\}\cup A\cup B$ of $V(G)$ where $A=N(u^*)$ and $B=V(G)\backslash N[u^*]$.
Thus,
\begin{align}\label{mainEq1-1}
\rho^2x_{u^*}=\sum_{u\in N(u^*)}\rho x_u&=\sum_{u\in N(u^*)}\sum_{v\in N(u)}x_v\nonumber\\
&=|A|x_{u^*}+\sum_{uv\in E(A)}(x_u+x_v)+\sum_{u\in B}d_A(u)x_u.
\end{align}

Next we establish two necessary claims.

%

\begin{claim}\label{clm:1-d_A(uinB)}
For a vertex $u$ in $B$, $d_A(u)\le 1$.
\end{claim}
\begin{proof}
This claim follows from the fact that $C_4\not\subseteq G$.
\end{proof}

Following the partition of $V(G)$, we give a refinement of $B$.
Let $B=B_1\cup B_2$ be a partition of $B$, such that $B_1=\{u\in B: d_B(u)=0\}$ and $B_2=B\backslash B_1$.
Then for a vertex $u\in B_1$, we have $d_A(u)=1$ from claim \ref{clm:1-d_A(uinB)}, and so $d(u)=1$.

\begin{claim}\label{clm:1-element of B_1}
For a vertex $u$ in $B_1$, $x_u\le\frac{1}{\rho}x_{u^*}$.
\end{claim}
\begin{proof}
Let $u\in B_1$, then we have $\rho x_u=\sum_{v\in N(u)}x_v\le x_{u^*}$.
The claim follows.
\end{proof}

Thus, from claim \ref{clm:1-element of B_1}, by (\ref{mainEq1-1}) we have
\begin{align}\label{mainEq2-1}
\rho^2x_{u^*}\le|A|x_{u^*}+\sum_{uv\in E(A)}(x_u+x_v)+\sum_{u\in B_1}\frac{1}{\rho}x_{u^*}+\sum_{u\in B_2}d_A(u)x_u.
\end{align}
Note that
\begin{align}\label{edgeA-B2}
e(A,B_2)\le 2e(B_2)=2e(B).
\end{align}



Note that $V(G)$ has the partition $\{u\}\cup A\cup B_1\cup B_2$.
Clearly $P_3\not\subseteq G[A]$ since $C_4\not\subseteq G$.
Then, from claim \ref{clm:1-element of B_1}, by (\ref{edgeA-B2}) we obtain
\begin{align*}
\rho\sum_{uv\in E(A)}(x_u+x_v)
&\le 2e(A)x_{u^*}+\sum_{uv\in E(A)}(x_u+x_v)+\sum_{u\in B_1}d_A(u)x_u+\sum_{u\in B_2}d_A(u)x_u\\
&\le 2e(A)x_{u^*}+\sum_{uv\in E(A)}(x_u+x_v)+\frac{e(A,B_1)}{\rho}x_{u^*}+e(A,B_2)x_{u^*}\\
&\le 2e(A)x_{u^*}+\sum_{uv\in E(A)}(x_u+x_v)+\frac{e(A,B_1)}{\rho}x_{u^*}+2e(B)x_{u^*}.
\end{align*}
It follows that
$$\sum_{uv\in E(A)}(x_u+x_v)\le \left(\frac{2e(A)+2e(B)}{\rho-1}+\frac{e(A,B_1)}{\rho(\rho-1)}\right)x_{u^*}.$$
This, together with (\ref{mainEq2-1}), indicates that
\begin{align*}
\rho^2 x_{u^*}&\le
|A|x_{u^*}+\left(\frac{2e(A)+2e(B)}{\rho-1}+\frac{e(A,B_1)}{\rho(\rho-1)}\right)x_{u^*}+ \frac{e(A,B_1)}{\rho}x_{u^*}
+\sum_{u\in B_2}d_A(u)x_u\\
&\le|A|x_{u^*}+\left(\frac{2e(A)+2e(B)}{\rho-1}+\frac{e(A,B_1)}{\rho(\rho-1)}\right)x_{u^*}+ \frac{e(A,B_1)}{\rho}x_{u^*}
+e(A,B_2)x_{u^*}\\
&=\left(|A|+\frac{2e(A)+2e(B)}{\rho-1}+\frac{e(A,B_1)}{\rho-1}+e(A,B_2)\right)x_{u^*}.
\end{align*}
That is, $\rho^2\le|A|+\frac{2e(A)+2e(B)}{\rho-1}+\frac{e(A,B_1)}{\rho-1}+e(A,B_2)$.

On the other hand, we know  that $\rho^2\ge m-k$.
Note that $m=|A|+e(A)+e(B)+e(A,B)=|A|+e(A)+e(B)+e(A,B_1)+e(A,B_2)$.
We have
\begin{align}\label{edge-1}
\rho^2\ge |A|+e(A)+e(B)+e(A,B_1)+e(A,B_2)-k.
\end{align}
Hence,
$$|A|+e(A)+e(B)+e(A,B_1)+e(A,B_2)-k\le |A|+\frac{2e(A)+2e(B)}{\rho-1}+\frac{e(A,B_1)}{\rho-1}+e(A,B_2),$$
which implies that
$$(\rho-3)e(A)+(\rho-3)e(B)+(\rho-2)e(A,B_1)\le k(\rho-1).$$
Since $m\ge (2k+3)^2+k+1$, we have $\rho\ge\sqrt{m-k}>2k+3$.
Thus, $e(B)\le \frac{\rho-1}{\rho-3}k<k+1$, and so $e(B)\le k$.

Therefore, for a vertex $u\in B_2$, we have $d(u)\le k+1$, and
$$\rho x_u=\sum_{v\in N(u)}x_v\le d(u)x_{u^*}\le (k+1)x_{u^*},$$
which follows that $x_u\le \frac{k+1}{\rho}x_{u^*}$.
Furthermore, we obtain
\begin{align*}
\rho\sum_{uv\in E(A)}(x_u+x_v)
&\le 2e(A)x_{u^*}+\sum_{uv\in E(A)}(x_u+x_v)+\sum_{u\in B_1}d_A(u)x_u+\sum_{u\in B_2}d_A(u)x_u\\
&\le 2e(A)x_{u^*}+\sum_{uv\in E(A)}(x_u+x_v)+\frac{e(A,B_1)}{\rho}x_{u^*}+\frac{(k+1) e(A,B_2)}{\rho}x_{u^*}.
\end{align*}
That is,
$$\sum_{uv\in E(A)}(x_u+x_v)\le \frac{1}{\rho-1}\left(2e(A)+\frac{e(A,B_1)}{\rho}+\frac{(k+1)e(A,B_2)}{\rho}\right)x_{u^*}.$$

By (\ref{mainEq2-1}), we have
\begin{align*}
\rho^2 x_{u^*}&\le
|A|x_{u^*}+\frac{1}{\rho-1}\left(2e(A)+\frac{e(A,B_1)}{\rho}+\frac{(k+1)e(A,B_2)}{\rho}\right)x_{u^*}+
\frac{e(A,B_1)}{\rho}x_{u^*}\\
&\ \ +\sum_{u\in B_2}d_A(u)x_u\\
&\le|A|x_{u^*}+\frac{1}{\rho-1}\left(2e(A)+\frac{e(A,B_1)}{\rho}+\frac{(k+1)e(A,B_2)}{\rho}\right)x_{u^*}+ \frac{e(A,B_1)}{\rho}x_{u^*}\\
&\ \ +\frac{(k+1)e(A,B_2)}{\rho}x_{u^*}\\
&=\left(|A|+\frac{2e(A)}{\rho-1}+\frac{e(A,B_1)}{\rho-1}+\frac{(k+1)e(A,B_2)}{\rho-1}\right)x_{u^*}.
\end{align*}
Combining this inequality with (\ref{edge-1}), we obtain
$$|A|+e(A)+e(B)+e(A,B_1)+e(A,B_2)-k\le |A|+\frac{2e(A)}{\rho-1}+\frac{e(A,B_1)}{\rho-1}+\frac{(k+1)e(A,B_2)}{\rho-1},$$
which implies that
\begin{align}\label{constraint-1}
(\rho-3)e(A)+(\rho-1)e(B)+(\rho-2)e(A,B_1)+(\rho-k-2)e(A,B_2)\le k(\rho-1).
\end{align}

Since $K_{1,m-k}\not\subseteq G$. Then $e(A)+e(B)+e(A,B_1)+e(A,B_2)=m-|A|\ge k+1$.
If $k=0$, then $(\rho-3)e(A)+(\rho-1)e(B)+(\rho-2)e(A,B_1)+(\rho-2)e(A,B_2)\le0$ by (\ref{constraint-1}).
So $\rho\le 3$.
Hence, $3\ge\rho\ge\sqrt{m-k}\ge \sqrt{(2k+3)^2+k+1-k}=\sqrt{10}$, a contradiction.

If $k\ge 1$, then by (\ref{constraint-1}) we have
$$e(A)+e(B)+e(A,B_1)+e(A,B_2)\le \frac{\rho-1}{\rho-k-2}k<k+1$$
since $\rho\ge\sqrt{m-k}\ge \sqrt{(k^2+2k+2)^2+k+1-k}>k^2+2k+2$.
This is a contradiction.

This completes the proof.
$\hfill\blacksquare$

\section{Proof of Theorem \ref{main-thm-2}}\label{section3}

\indent \, \quad Notations appeared in this section are the same as those in section \ref{section2}.
For a graph $G$, if $\emph{\textbf{x}}$ is a unit eigenvector of $Q(G)$ corresponding to $q(G)$
with coordinate $x_i$ corresponding to the vertex $v_i$ of $G$,
by the well-known Courant-Fisher theorem, then we have
\begin{align}\label{mainEq1-2}
q(G)=\max_{\|\textbf{y}\|_2=1}\textbf{y}^TQ(G)\textbf{y}=\sum_{v_iv_j\in E(G)}(x_i+x_j)^2.
\end{align}
Note that the formulate $Q(G)\emph{\textbf{x}}=q(G)\emph{\textbf{x}}$ implies that $\big(q(G)I-D(G)\big)\emph{\textbf{x}}=A(G)\emph{\textbf{x}}$.
Then for a vertex $u\in V(G)$, we have
\begin{align}\label{mainEq2-2}
\big(q(G)-d(u)\big)x_u=\sum_{v\in N(u)}x_v.
\end{align}

\vskip0.1in
\noindent
{\bf Proof of Theorem \ref{main-thm-2}.}
We prove theorem \ref{main-thm-2} by way of contradiction.
Suppose that $G$ is the extremal graph with the maximum signless Laplacian spectral radius
among graphs $H$ of size $m\ge\max\{\frac{1}{2}k^2+6k+3, 7k+25\}$ and $K_{1,m-k}\not\subseteq H$.
Then $q(G)\ge m-k+1$.
Let $\emph{\textbf{x}}$ be a nonnegative unit eigenvector of $Q(G)$ corresponding to $q(G)$,
and $u^*$ be a vertex of $G$ with $x_{u^*}=\max\{x_i:v_i\in V(G)\}$.
For simplification, write $q$ by $q(G)$.

Denote by
$$W=\left\{u\in V(G): x_u\ge \frac{1}{2}x_{u^*}\right\}.$$
Note that $u^*\in W$, and so $|W|\ge 1$. We prove the following claim.

\begin{claim}\label{clm:2-W}
$|W|=1$.
\end{claim}
\begin{proof}
For a vertex $u\in W$, we know $x_u\ge \frac{1}{2}x_{u^*}$.
Then, by (\ref{mainEq2-2}),
$$\big(q-d(u)\big)\frac{1}{2}x_{u^*}\le\big(q-d(u)\big)x_u=\sum_{v\in N(u)}x_v\le d(u)x_{u^*},$$
which follows that $d(u)\ge\frac{1}{3}q$.

Since $q\ge m-k+1$ and $m\ge 7k+24$. Then
$$2m\ge\sum_{u\in W}d(u)\ge\frac{1}{3}q|W|\ge \frac{1}{3}(m-k+1)|W|,$$
that is, $|W|\le\frac{6m}{m-k+1}<7$. Thus, $|W|\le 6$.

Now we can improve the lower bound that $d(u)\ge\frac{1}{3}q$ for $u\in W$.
By (\ref{mainEq2-2}), we obtain that
\begin{align*}
\big(q-d(u^*)\big)x_{u^*}=\sum_{v\in N(u^*)}x_v
&=\sum_{v\in N(u^*)\cap W}x_v+\sum_{v\in N(u^*)\backslash W}x_v\nonumber\\
&\le (|W|-1)x_{u^*}+(d(u^*)-|W|+1)\frac{1}{2}x_{u^*}\\
&=\frac{1}{2}(d(u^*)+|W|-1)x_{u^*},
\end{align*}
which follows that
\begin{align}\label{Eq2-u'}
d(u^*)\ge\frac{2}{3}q-\frac{1}{3}|W|+\frac{1}{3}\ge\frac{2}{3}q-\frac{5}{3}.
\end{align}

Assume that $|W|\ge 2$.
For a vertex $u\in W\backslash \{u^*\}$, we obtain that
\begin{align*}
\big(q-d(u)\big)x_{u}=\sum_{v\in N(u)}x_v
&=\sum_{v\in N(u)\cap W}x_v+\sum_{v\in N(u)\backslash W}x_v\nonumber\\
&\le (|W|-1)x_{u^*}+(d(u)-|W|+1)\frac{1}{2}x_{u^*}\\
&=\frac{1}{2}(d(u)+|W|-1)x_{u^*}.
\end{align*}
On the other hand, we have $(q-d(u))x_{u}\ge\frac{1}{2}(q-d(u))x_{u^*}$.
Hence, $\frac{1}{2}(d(u)+|W|-1)\ge \frac{1}{2}(q-d(u)),$
that is,
\begin{align}\label{Eq2-u}
d(u)\ge\frac{q}{2}-\frac{5}{2}.
\end{align}

Combining (\ref{Eq2-u'}) and (\ref{Eq2-u}), we have
$$m+1\ge d(u^*)+d(u)\ge \frac{2}{3}q-\frac{5}{3}+\frac{q}{2}-\frac{5}{2}=
\frac{7}{6}q-\frac{25}{6}\ge \frac{7}{6}(m-k+1)-\frac{25}{6}.$$
Hence, $m\le 7k+24$, which contradicts the fact that $m\ge 7k+25$.
Thus, $|W|\le 1$, and so $|W|=1$ since $|W|\ge 1$.
\end{proof}



From claim \ref{clm:2-W}, we have $W=\{u^*\}$.
Thus, for two vertices $u,v\in V(G)\backslash \{u^*\}$, it has $x_u+x_v<x_{u^*}$.

We assert that $d(u^*)=m-k-1$. On the contrary, suppose that $d(u^*)\le m-k-2$.
Then there is an edge, says $u_1u_2\in E(G)$, such that $u\notin \{u_1,u_2\}$.
Let $G'$ be the graph obtained from $G$ by deleting the edge $u_1u_2$ and
attaching a pendent vertex $u_0$ to $u^*$, and $\emph{\textbf{x}}'$ be a vector with
\begin{equation*}
x'_w=\left\{
\begin{aligned}
&x_w,&\textrm{if}&\ w\in V(G);\\
&0, & \textrm{if}&\ w=u_0.
\end{aligned}
\right.
\end{equation*}

Note that $\|\emph{\textbf{x}}'\|_2=1$.
By (\ref{mainEq1-2}), we have
\begin{align*}
q(G')-q(G)&\ge \sum_{uv\in E(G')}(x'_u+x'_v)^2-\sum_{uv\in E(G)}(x_u+x_v)^2\\
&=(x_{u^*}+0)^2-(x_{u_1}+x_{u_2})^2>0.
\end{align*}
Since $K_{1,m-k}\not\subseteq G'$.
This deduces a contradiction to the maximality of $G$. Thus, we have $d(u^*)=m-k-1$.

For a vertex $u\in V(G)\backslash\{u^*\}$, we have $d(u)\le k+2$.
Then, from (\ref{mainEq2-2}), we have
\begin{align*}
\big(q-d(u)\big)x_u=\sum_{v\in N(u)}x_v\le x_{u^*}+(d(u)-1)\frac{1}{2}x_{u^*},
\end{align*}
which follows that
\begin{align}\label{Eq2-element-u}
x_u\le \frac{d(u)+1}{2(q-d(u))}x_{u^*}\le \frac{k+3}{2(q-k-2)}x_{u^*}.
\end{align}

We can further improve the lower bound in (\ref{Eq2-element-u}).
Similarly, by (\ref{Eq2-element-u}) we have
\begin{align*}
\big(q-d(u)\big)x_u=\sum_{v\in N(u)}x_v\le x_{u^*}+(d(u)-1)\frac{k+3}{2(q-k-2)}x_{u^*},
\end{align*}
which implies that
\begin{align}\label{Eq2-element-uu}
x_u\le \left(\frac{1}{q-d(u)}+\frac{d(u)-1}{q-d(u)}\frac{k+3}{2(q-k-2)}\right)x_{u^*}\le
\left(\frac{1}{q-k-2}+\frac{(k+1)(k+3)}{2(q-k-2)^2}\right)x_{u^*}.
\end{align}

Recall that $q\ge m-k+1$ and $d(u^*)=m-k-1$.
By (\ref{mainEq2-2}), we obtain
\begin{align*}
2x_{u^*}\le \big(q-d(u^*)\big)x_{u^*}=\sum_{u\in N(u^*)}x_u\le d(u^*)\left(\frac{1}{q-k-2}+\frac{(k+1)(k+3)}{2(q-k-2)^2}\right)x_{u^*}.
\end{align*}
So
\begin{align*}
(m-k-1)\left(\frac{1}{q-k-2}+\frac{(k+1)(k+3)}{2(q-k-2)^2}\right)\ge 2.
\end{align*}
On the other hand, we may check that
\footnotesize{
\begin{align*}
(m-k-1)\left(\frac{1}{q-k-2}+\frac{(k+1)(k+3)}{2(q-k-2)^2}\right)
&\le (m-k-1)\left(\frac{1}{m-2k-1}+\frac{(k+1)(k+3)}{2(m-2k-1)^2}\right)\\
&=(m-2k-1+k)\left(\frac{1}{m-2k-1}+\frac{(k+1)(k+3)}{2(m-2k-1)^2}\right)\\
&=1+\frac{(k^2+6k+3)m-(k^3+9k^2+9k+3)}{2(m-2k-1)^2}\\
&<2,
\end{align*}} \normalsize
where the last inequality holds due to the fact that $m\ge \frac{1}{2}k^2+6k+3$.
This deduces a contradiction.

This completes the proof.
$\hfill\blacksquare$

\section{Concluding remarks}\label{concluding work}

\indent \, \quad Nikiforov \cite{Niki07} showed that if $G$ contains no $C_4$ then $\rho(G)\le \rho(F_n)$, where
$F_n$ is the friendship graph of odd order $n$, with equality if and only if $G=F_n$.
In the same paper (also see \cite{Niki09}), Nikiforov posed a conjecture that for even $n$,
if $G$ contains no $C_4$ then $\rho(G)\le \rho(F'_n)$, where $F'_n$ is obtained from $F_{n-1}$ by
attaching a new vertex to the unique vertex of maximum degree,
with equality if and only if $G=F'_n$.
The conjecture was confirmed by Zhai and Wang in \cite{ZW12}.

It is easy to check that
$$\rho(F_n)=\frac{1+\sqrt{4(n-1)+1}}{2}.$$
Due to a well-known fact that $\rho(G)\ge\frac{2m}{n}$ for a graph $G$ of order $n$ and size $m$,
if $G$ contains no $C_4$,
then we have
$$\frac{2m}{n}\le \rho(G)\le \rho(F_n)=\frac{1+\sqrt{4(n-1)+1}}{2}.$$
That is,
$$m\le\frac{n\left(1+\sqrt{4(n-1)+1}\right)}{4},$$
which is a classic upper bound of the Tur\'an number for $C_4$ by Reiman \cite{R58}.

One can see that from Nikiforov's result on odd $n$ (resp., Zhai-Wang's result on even $n$),
if $\rho(G)\ge \rho(F_n)$ (resp., $\rho(G)\ge \rho(F'_n)$), then $C_4\subseteq G$ or $K_{1,n-1}\subseteq G$.
Motivated by this property, we provide a natural conjecture in terms of the maximum degree as following.

\begin{conj}\label{conj}
Let $s\ge 1$ be an integer and $n\ge f(s)$, where $f(s)$ is a function on $s$.
If $G$ is a graph of order $n$ and
$$\rho(G)\ge \frac{1+\sqrt{4(n-s)+1}}{2},$$
then $K_{1,n-s}\subseteq G$ or $C_4\subseteq G$.
\end{conj}

Nikiforov's theorem confirmed Conjecture \ref{conj} for $s=1.$
Indeed, Conjecture \ref{conj} provides a spectral method to pursue the Tur\'an number for $C_4$.

\end{document}